\documentclass[12pt,reqno]{amsart}
\usepackage{latexsym,amsmath,amsfonts,amssymb,amsthm}
\def\pmod #1{\ ({\rm{mod}}\ #1)}
\def\Z{\mathbb Z}

\def\1{{\mathbf 1}}

\def\pmod #1{\ ({\rm{mod}}\ #1)}
\def\mod #1{\ \rm{mod}\ #1}

\def\floor #1{\left\lfloor{#1}\right\rfloor}

\theoremstyle{plain}
\newtheorem{Thm}{Theorem}
\newtheorem{Lem}{Lemma}

\theoremstyle{definition}
\newtheorem*{Ack}{Acknowledgment}
\theoremstyle{remark}

\begin{document}

\title{$p$-adic supercongruences conjectured by Sun}
\author{Yong Zhang}
\email{yongzhang@njit.edu.cn}
\address{Department of Mathematics and Physics, Nanjing Institute of Technology,
Nanjing 211167, People's Republic of China}
\thanks{
The work is supported by National Natural Science Foundation of China (Grant No.11226277) and  Natural Science Foundation of  Nanjing Institute of Technology (No. CKJB201807).}
\begin{abstract}In this paper we prove three results conjectured by Z.-W. Sun.
Let $p$ be an odd prime and let $h\in \Z$ with $2h-1\equiv0\pmod{p^{}}$. For $a\in\Z^{+}$ and $p^a>3$, we show that
\begin{align}\notag
\sum_{k=0}^{p^a-1}\binom{hp^a-1}{k}\binom{2k}{k}\bigg(-\frac{h}{2}\bigg)^k\equiv0\pmod{p^{a+1}}.
\end{align}
Also, for any $n\in \Z^{+}$ we have 
\begin{align}
\notag
\nu_{p}\bigg(\sum_{k=0}^{n-1}\binom{hn-1}{k}\binom{2k}{k}\bigg(-\frac{h}{2}\bigg)^k\bigg)\geq\nu_{p}(n)\notag,
\end{align}
where  $\nu_p(n)$ denotes the $p$-adic order of $n$.
 For any integer $m\not\equiv 0\pmod{p^{}}$ and positive integer $n$, we have
\begin{align*}
\frac{1}{pn}\bigg(\sum_{k=0}^{pn-1}\binom{pn-1}{k}\frac{\binom{2k}{k}}{(-m)^k}-\bigg(\frac{m(m-4)}{p}\bigg)\sum_{k=0}^{n-1}\binom{n-1}{k}\frac{\binom{2k}{k}}{(-m)^k}\bigg)\in \Z_{p},
\end{align*}
where $(\frac{.}{})$ is the Legendre symbol and $\Z_p$ is the ring of $p$-adic integers.
\end{abstract}
\maketitle

\section{Introduction}
\setcounter{equation}{0}
\setcounter{Thm}{0}
\setcounter{Lem}{0}
\setcounter{Cor}{0}
\setcounter{Conj}{0}
Let $p$ be an odd prime. In 2006, Pan and Sun
\cite{PS} proved the congruence
\begin{align*}
\sum_{k=0}^{p-1}\binom{2k}{k+d}\equiv \bigg(\frac{p-d}{3}\bigg)\pmod{p^{}}\ {\rm for}\ d=0,\ldots,p
\end{align*}
via a curious combinatorial identity.
For any positive integer $a$, later Sun and Tauraso \cite{ZWSRT} established the following general result 
\begin{align*}\sum_{k=0}^{p^a-1}\binom{2k}{k}\equiv\bigg(\frac{p^a}{3}\bigg)\pmod{p^{2}}.
\end{align*}

 Let $A, B\in \Z$. The Lucas sequence  ${u_n=u_n(A, B)}\ ({n\geq 0})$ is given by $$u_0=0,\ \  u_1=1\ \  {\rm and}\ \ u_{n+1}=Au_n-Bu_{n-1}\ (n\geq 1).$$ If $p$ is an odd prime not dividing $B$, then it is known that $p\mid u_{p-(\frac{\Delta}{p})}$
(see, e.g., \cite{ZWS3}). For a non-zero integer $n$ and a prime $p$, let $\nu_p(n)$ denote the $p$-adic valuation (or $p$-adic order) of $n$, i.e., $\nu_p(n)$ is the largest integer such that $p^{\nu_p(n)}\mid n$, especially $\nu_p(0)=+\infty$ and we define $\nu_p(\frac{m}{n})=\nu_p(m)-\nu_p(n)$ for rational number $\frac{m}{n}$.

 In 2011, Sun \cite{ZWS} proved that, for any nonzero integer
$m$ and odd prime $p$ with $p\nmid m$, there holds
\begin{align}\label{j7k7}
\sum_{k=0}^{p-1}\frac{\binom{2k}{k}}{m^k}\equiv\bigg(\frac{\Delta}{p}\bigg)+u_{p-(\frac{\Delta}{p})}(m-2,1)\pmod{p^{2}},
\end{align}
where $\Delta=m(m-4)$.
As a common extension of (\ref{j7k7}), Sun \cite{ZWS17} showed that 
\begin{align}\label{777477}\frac{1}{pn}\bigg(\sum_{k=0}^{pn-1}\frac{\binom{2k}{k}}{m^k}-\bigg(\frac{\Delta}{p}\bigg)\sum_{k=0}^{n-1}\frac{\binom{2k}{k}}{m^k}\bigg)\in \Z_{p}
\end{align}
and furthermore
\begin{align}\label{77747}\frac{1}{n}\bigg(\sum_{k=0}^{pn-1}\frac{\binom{2k}{k}}{m^k}-\bigg(\frac{\Delta}{p}\bigg)\sum_{k=0}^{n-1}\frac{\binom{2k}{k}}{m^k}\bigg)\equiv \frac{\binom{2n}{n}}{2m^{n-1}}u_{p-(\frac{\Delta}{p})}(m-2,1)\pmod{p^{{2}
}}.
\end{align}

 Let $p$ be an odd prime and let $m$ be an integer with $p\nmid m.$ One can easily get the following
formulae:  
\begin{align}\label{7774777}\sum_{k=0}^{p-1}\binom{p-1}{k}(-1)^k\frac{\binom{2k}{k}}{m^k}\equiv\sum_{k=0}^{p-1}\frac{\binom{2k}{k}}{m^k}\pmod{p^{{}}}.
\end{align} 
It looks like the left-hand side of (\ref{7774777}) has some connection with  the  right-hand side.
Motivated by (\ref{j7k7}) and (\ref{7774777}), Sun \cite{ZWS2} determined the sum $\sum_{k=0}^{p^a-1}\binom{hp^a-1}{k}\binom{2k}{k}/m^k$ modulo $p^2$, where $h$ is a $p$-adic integer and $m\in \Z$ with $p\nmid m$. For example, if $h\not\equiv 0\pmod{p^{{}}}$ and ($(2h\not\equiv 1\pmod{p})$ or $p^a>3$), then 
\begin{align}\notag
&\sum_{k=0}^{p^a-1}\binom{hp^a-1}{k}\binom{2k}{k}\bigg(-\frac{h}{2}\bigg)^k\\\label{j7k7y7}&\equiv\bigg(\frac{1-2h}{p^a}\bigg)\bigg(1+h\bigg(\bigg(4-\frac{2}{h}\bigg)^{p-1}-1\bigg)\bigg)\pmod{p^{{2}}}.
\end{align}

It is natural to ask whether there exists the supercongruenc as (\ref{j7k7y7}) modulo the higher
powers of $p$
in the case $2h-1\equiv0\pmod{p^{}}$ and $p^a>3$. Sun \cite{ZWS2} managed to investigate the above case and made the following conjecture. The first aim of this paper is to prove the conjectured results.
\begin{Thm}
Let $p$ be an odd prime and let $h\in \Z$ with $2h-1\equiv0\pmod{p^{}}$. If $a\in\Z^{+}$ and $p^a>3$, then
\begin{align} \label{jiankun11}
\sum_{k=0}^{p^a-1}\binom{hp^a-1}{k}\binom{2k}{k}\bigg(-\frac{h}{2}\bigg)^k\equiv0\pmod{p^{a+1}}.
\end{align}
Also, for any $n\in \Z^{+}$ we have 
\begin{align}
\label{jiankun12}
\nu_{p}\bigg(\sum_{k=0}^{n-1}\binom{hn-1}{k}\binom{2k}{k}\bigg(-\frac{h}{2}\bigg)^k\bigg)\geq\nu_{p}(n).
\end{align}
\end{Thm}
On the other hand, based on (\ref{777477}) and (\ref{7774777}), Sun \cite{ZWS17} conjectured the corresponding result  with $p\nmid m$ .The second aim of this paper is to show the following result.
\begin{Thm}Let $p$ be an odd prime and let $\Delta=m(m-4)$. For any integer $m\not\equiv 0\pmod{p^{}}$ and positive integer $n$, we have
\begin{align}\label{jiankun1177}
\frac{1}{pn}\bigg(\sum_{k=0}^{pn-1}\binom{pn-1}{k}\frac{\binom{2k}{k}}{(-m)^k}-\bigg(\frac{\Delta}{p}\bigg)\sum_{k=0}^{n-1}\binom{n-1}{k}\frac{\binom{2k}{k}}{(-m)^k}\bigg)\in \Z_{p}.
\end{align}
\end{Thm}

The remainder of the paper is organized as follows. In the next section, we give some lemmas. The proofs of Theorems 1.1 and 1.2 will be given in Section 3.

\section{Some Lemmas}
\setcounter{equation}{0}
\setcounter{Thm}{0}
\setcounter{Lem}{0}
\setcounter{Cor}{0}
\setcounter{Conj}{0}
In the following section, for an assertion $A$ we adopt the notation:
$$[A]=\begin{cases}1,&\text{if A holds,} \\
0,&\text{otherwise}.\end{cases} $$
We know that $[m=n]$ coincides with the Kronecker symbol $\delta_{m,n}$.
\begin{Lem}
\label{lzhang}
Let $n, k, \alpha$ be positive integers and $p$ be a prime. Then
\begin{equation}\label{y71}
\binom{p^{\alpha}n-1}{k}\equiv \binom{p^{\alpha-1}n-1}{\floor {\frac{k}{p}}}(-1)^{k-\floor {\frac{k}{p}}}(1-np^{\alpha}\sum_{\substack{j=1\\p\nmid j}}^{k}\frac{1}{j})\pmod{p^{2{\alpha}}}.
\end{equation}
\end{Lem}
\begin{proof}
Note that
\begin{align}\notag
\binom{p^\alpha n-1}{k}&=\prod_{j=1}^{k}\frac{p^{\alpha} n-j}{j}=\binom{p^{{\alpha}-1}n-1}{\floor {\frac{k}{p}}}\prod_{\substack{j=1\\p\nmid j}}^{k}\frac{p^{\alpha} n-j}{j}\\\notag
&\equiv\binom{p^{{\alpha}-1}n-1}{\floor {\frac{k}{p}}}(-1)^{k-\floor {\frac{k}{p}}}(1-\sum_{\substack{j=1\\p\nmid j}}^{k}\frac{p^{\alpha} n}{j})\pmod{p^{2{\alpha}}}.
\end{align}
This proves (\ref{y71}).
\end{proof}

\begin{Lem}Let $p$ be an odd prime. Then, for any integers $a, b$ and positive integers $r, s,$ we have
\begin{align}\label{yzhp6}
\binom{p^ra}{p^sb}/\binom{p^{r-1}a}{p^{s-1}b}\equiv 1(\mod p^{r+s+min\{r, s\}-\delta_{p,3}}).
\end{align}
\end{Lem}
This lemma is a well-known congruence due to R. Osburn, B. Sahu and A. Straub, see, e.g., \cite[(19)]{ROBSAS}.

The following curious result is due to Sun \cite{ZWS7}:
\begin{Lem}[See {\cite[Theorem 1.1]{ZWS7}}]
\label{y5}
Let $m\in \Z$ and $n\in \Z^{+}$. Suppose that $p$ is an odd prime dividing $m-4$. Then
\begin{align}\label{yzhp1}
\nu_{p}\bigg(\sum_{k=0}^{n-1}\frac{\binom{2k}{k}}{m^k}\bigg)\geq\nu_{p}(n)\ \ and \ \ \nu_{p}\bigg(\sum_{k=0}^{n-1}\binom{n-1}{k}(-1)^k\frac{\binom{2k}{k}}{m^k}\bigg)\geq\nu_{p}(n).
\end{align}
Furthermore,
\begin{align}\label{yzhp2}
\frac{1}{n}\sum_{k=0}^{n-1}\frac{\binom{2k}{k}}{m^k}\equiv \frac{\binom{2n-1}{n-1}}{4^{n-1}}+\delta_{p,3}[3\mid n]\frac{m-4}{3}\binom{\frac{2n}{3^{\nu_{3}(n)}}-1}{\frac{n}{3^{\nu_{3}(n)}}-1}\pmod{p^{\nu_{p}(m-4)}}
\end{align}
and also
\begin{align*}
\frac{1}{n}\sum_{k=0}^{n-1}\binom{n-1}{k}(-1)^k\frac{\binom{2k}{k}}{m^k}\equiv\frac{C_{n-1}}{4^{n-1}}\pmod{p^{\nu_{p}(m-4)-\delta_{p,3}}},
\end{align*}
where $C_{k}$ denote the Catalan number $\frac{1}{k+1}\binom{2k}{k}=\binom{2k}{k}-\binom{2k}{k+1}$. Thus, for $a\in \Z^{+}$ we have 
\begin{align}\label{yzhp4}
\frac{1}{p^a}\sum_{k=0}^{p^a-1}\frac{\binom{2k}{k}}{m^k}\equiv1+\delta_{p,3}\frac{m-4}{3}\equiv\frac{m-1}{3}\pmod{p}.
\end{align}
\end{Lem}
\begin{Lem}
Let $p$ be an odd prime and let $h\in \Z$ with $2h-1\equiv0\pmod{p^{}}$. Let $l, {\alpha}$ be nonnegative integers. If $p\geq 5$, then we have
\begin{align}\label{yzhp5}
\sum_{\floor {\frac{k}{p^{\alpha}}}=l}\binom{2k}{k}\bigg(\frac{h}{2}\bigg)^k\equiv p^{{\alpha}}\frac{\binom{2l}{l}}{4^l}\pmod{p^{{\alpha}
+1}}.
\end{align}
If $p=3$, then
\begin{align}\label{yzhp15}
\sum_{\floor {\frac{k}{p^{\alpha}}}=l}\binom{2k}{k}\bigg(\frac{h}{2}\bigg)^k\equiv 0\pmod{p^{{\alpha}
}}.
\end{align}
If $p=3$ and $l\equiv 1\pmod{3},$ then 
\begin{align}\label{yzhp8}
\sum_{\floor {\frac{k}{p^{\alpha}}}=l}\binom{2k}{k}\bigg(\frac{h}{2}\bigg)^k\equiv p^{{\alpha}}\frac{\binom{2l}{l}}{4^l}+2[{\alpha}\geq 1]p^{{\alpha}-1}\frac{(1-2h)\binom{2l}{l}}{h}\pmod{p^{{\alpha}
+1}}.
\end{align}
\end{Lem}
\begin{proof}
Observe that
\begin{align*}
\sum_{\floor {\frac{k}{p^{\alpha}}}=l}\binom{2k}{k}\bigg(\frac{h}{2}\bigg)^k&=
\sum_{k=0}^{p^{\alpha}l+p^{\alpha}-1}\binom{2k}{k}\bigg(\frac{h}{2}\bigg)^k-\sum_{k=0}^{p^{\alpha}l-1}\binom{2k}{k}\bigg(\frac{h}{2}\bigg)^k.
\end{align*}
Since $\frac{2}{h}-4=\frac{2(1-2h)}{h}\equiv 0\pmod{p}$ and $l\equiv 1\pmod{3}$, by (\ref{yzhp2}) we obtain that
\begin{align}\notag
&\sum_{\floor {\frac{k}{p^{\alpha}}}=l}\binom{2k}{k}\bigg(\frac{h}{2}\bigg)^k\\\notag&\equiv (p^{\alpha}l+p^{\alpha})\bigg(\frac{\binom{(2l+2)p^{\alpha}}{(l+1)p^{\alpha}}}{4^{p^{\alpha}l+p^{\alpha}-1}2}+\delta_{p,3}[3\mid (p^{\alpha}l+p^{\alpha})]\frac{2(1-2h)}{3h}\binom{\frac{2(p^{\alpha}l+p^{\alpha})}{3^{\nu_{3}(p^{\alpha}l+p^{\alpha})}}-1}{\frac{p^{\alpha}l+p^{\alpha}}{3^{\nu_{3}(p^{\alpha}l+p^{\alpha})}}-1}\bigg)\\\label{yzhp999}
&\ \ -p^{\alpha}l\bigg(
\frac{\binom{2p^{\alpha}l}{p^{\alpha}l}}{4^{p^{\alpha}l-1}2}+\delta_{p,3}[3\mid (p^{\alpha}l)]\frac{2(1-2h)}{3h}\binom{\frac{2p^{\alpha}l}{3^{\nu_{3}(p^{\alpha}l)}}-1}{\frac{p^{\alpha}l}{3^{\nu_{3}(p^{\alpha}l)}}-1}\bigg)\pmod{p^{{\alpha}
+1}}.
\end{align}
When $p\geq 5$, then by (\ref{yzhp6}) and (\ref{yzhp999}) we get
\begin{align*}
\sum_{\floor {\frac{k}{p^{\alpha}}}=l}\binom{2k}{k}\bigg(\frac{h}{2}\bigg)^k&\equiv p^{\alpha}(l+1)\frac{\binom{2l+2}{l+1}}{4^{l}2}-p^{\alpha}l\frac{\binom{2l}{l}}{4^{l-1}2}\\
&=p^{\alpha}(2l+1)\frac{\binom{2l}{l}}{4^{l}}-2p^{\alpha}l\frac{\binom{2l}{l}}{4^{l}}=p^{{\alpha}}\frac{\binom{2l}{l}}{4^l}\pmod{p^{{\alpha}
+1}}.
\end{align*}
Thus (\ref{yzhp5}) is proved. (\ref{yzhp15}) is eazily deduced from (\ref{yzhp999}).

Finally we will prove (\ref{yzhp8}). (\ref{yzhp8}) is trivial when $\alpha=0.$  
Now we may assume $\alpha \geq 1$.
With the help of (\ref{yzhp6}) and (\ref{yzhp999}), for any nonnegative integer $l$ with  
$l\equiv 1\pmod{3}$ we have
\begin{align*}
\sum_{\floor {\frac{k}{3^{\alpha}}}=l}\binom{2k}{k}\bigg(\frac{h}{2}\bigg)^k&\equiv 3^{\alpha}(l+1)\bigg(\frac{\binom{2l+2}{l+1}}{4^{l}2}+\frac{2(1-2h)}{3h}\binom{2l+1}{l+1}\bigg)\\\notag&\ \ -3^{\alpha}l\bigg(\frac{\binom{2l}{l}}{4^{l-1}2}+\frac{1-2h}{3h}\binom{2l}{l}\bigg)\\
&\equiv3^{{\alpha}}\frac{\binom{2l}{l}}{4^l}+23^{{\alpha}-1}\frac{(1-2h)\binom{2l}{l}}{h}\pmod{3^{{\alpha}
+1}}.
\end{align*}
This concludes the proof. 
\end{proof}
\begin{Lem}
Let $p$ be an odd prime and $l, s$ be nonnegative integers. Let $\Delta=m(m-4)$. For any integer $m\not\equiv 0\pmod{p^{
}}$, we have
\begin{align}\notag
&\sum_{k=p^sl}^{p^sl+p^s-1}\frac{\binom{2k}{k}}{m^k}-\bigg(\frac{\Delta}{p}\bigg)\sum_{k=p^{s-1}l}^{p^{s-1}l+p^{s-1}-1}\frac{\binom{2k}{k}}{m^k}\\\label{7747}&
\equiv p^{{s-1}}\frac{\binom{2l}{l}}{m^l}((2-\frac{m}{2})l+1)u_{p-(\frac{\Delta}{p})}(m-2,1)\pmod{p^{{s}
+1}}.
\end{align}
\end{Lem}
\begin{proof}
The proof is very similar with (\ref{yzhp5}). Clearly,
\begin{align*}\sum_{k=p^sl}^{p^sl+p^s-1}\frac{\binom{2k}{k}}{m^k}=\sum_{k=0}^{p^sl+p^s-1}\frac{\binom{2k}{k}}{m^k}-\sum_{k=0}^{p^sl-1}\frac{\binom{2k}{k}}{m^k}
\end{align*}
Substituting $n=p^{s-1}(l+1)$ and $n=p^{s-1}l$ in (\ref{77747}), we obtain (\ref{7747}).
\end{proof}
\begin{Lem}
Let $p$ be an odd prime and $h\in \Z$ with $2h-1\equiv0\pmod{p^{}}$. Let $a$ be a positive integer. Then
\begin{align}\label{yzhp10}
\sum_{k=0}^{p^a-1}{\binom{2k}{k}}\bigg(\frac{h}{2}\bigg)^k\sum_{\substack{j=1\\p\nmid j}}^{k}\frac{1}{j}\equiv \begin{cases}0\pmod{p^{}},&\text{if $a\geq 2$}, \\
2\pmod{p^{
}},&\text{if $a=1$.}\end{cases}
\end{align}
\end{Lem}
\begin{proof}
Note that
\begin{align}\label{yzhp11}
\sum_{k=0}^{p^a-1}\binom{2k}{k}\bigg(\frac{h}{2}\bigg)^k\sum_{\substack{j=1\\p\nmid j}}^{k}\frac{1}{j}=\sum_{s=0}^{p^{a-1}-1}\sum_{t=0}^{p-1}\binom{2(ps+t)}{ps+t}\bigg(\frac{h}{2}\bigg)^{ps+t}\sum_{\substack{j=1\\p\nmid j}}^{ps+t}\frac{1}{j}
\end{align}
and
$$\sum_{\substack{j=1\\p\nmid j}}^{ps+t}\frac{1}{j}=\sum_{\substack{j=1\\p\nmid j}}^{ps-1}\frac{1}{j}+\sum_{j=ps+1}^{ps+t}\frac{1}{j}=\sum_{\substack{j=1\\p\nmid j}}^{ps-1}\frac{1}{ps-j}+\sum_{j=1}^{t}\frac{1}{ps+j}\equiv \sum_{j=1}^{t}\frac{1}{j}\pmod{p^{
}}.$$
With the help of Lucas' theorem  (cf. \cite[ p. 44]{RPS}), it follows that 
\begin{align}\notag
\sum_{k=0}^{p^a-1}\binom{2k}{k}\bigg(\frac{h}{2}\bigg)^k\sum_{\substack{j=1\\p\nmid j}}^{k}\frac{1}{j}&\equiv\sum_{s=0}^{p^{a-1}-1}\sum_{t=1}^{\frac{p-1}{2}}\binom{2s}{s}\binom{2t}{t}\bigg(\frac{h}{2}\bigg)^{s+t}\sum_{j=1}^{t}\frac{1}{j}\\\notag
&\ \ +\sum_{s=0}^{p^{a-1}-1}\sum_{t=\frac{p+1}{2}}^{p-1}\binom{2s+1}{s}\binom{2t-p}{t}\bigg(\frac{h}{2}\bigg)^{s+t}\sum_{j=1}^{t}\frac{1}{j}\\\label{yzhp12}
&\equiv \sum_{s=0}^{p^{a-1}-1}\binom{2s}{s}\bigg(\frac{h}{2}\bigg)^{s}\sum_{t=1}^{\frac{p-1}{2}}\binom{2t}{t}\bigg(\frac{h}{2}\bigg)^{t}\sum_{j=1}^{t}\frac{1}{j}\pmod{p^{
}}.
\end{align}
Since  $p\mid (\frac{2}{h}-4)$, in view of (\ref{yzhp1}), then 
\begin{align}\label{yzhp137}\sum_{s=0}^{p^{a-1}-1}\binom{2s}{s}\bigg(\frac{h}{2}\bigg)^{s}\equiv 0\pmod{p^{a-1
}}.
\end{align}
If $a\geq 2$, combining (\ref{yzhp12}) and (\ref{yzhp137}), then (\ref{yzhp10}) holds.

When $a=1.$ In fact, for $k=0,\ldots,\frac{p-1}{2}$ we clearly have

$$\binom{2k}{k}=\binom{-\frac{1}{2}}{k}(-4)^k\equiv \binom{\frac{p-1}{2}}{k}(-4)^k\pmod{p^{
}}.$$ 
Observing that $2h\equiv 1\pmod{p},$ then we obtain 
\begin{align}\notag\sum_{t=1}^{\frac{p-1}{2}}\binom{2t}{t}\bigg(\frac{h}{2}\bigg)^{t}\sum_{j=1}^{t}\frac{1}{j}&\equiv\sum_{t=1}^{\frac{p-1}{2}}\binom{\frac{p-1}{2}}{t}(-2h)^{t}\sum_{j=1}^{t}\frac{1}{j}\\\label{jianjiankunkunyong}&\equiv\sum_{j=1}^{\frac{p-1}{2}}\frac{(-1)^j}{j}\sum_{t=j}^{\frac{p-1}{2}}\binom{\frac{p-1}{2}}{t}\binom{-1}{t-j}\pmod{p^{
}}.
\end{align}
Recall that the Chu-Vandermonde identity (See, e.g., \cite[ p. 169]{RLGDEKOP})
$$\sum_{k=0}^{n}\binom{x}{k}\binom{y}{n-k}=\binom{x+y}{n}.$$
Thus,
\begin{align}\notag\sum_{j=1}^{\frac{p-1}{2}}\frac{(-1)^j}{j}\sum_{t=j}^{\frac{p-1}{2}}\binom{\frac{p-1}{2}}{t}\binom{-1}{t-j}&=\sum_{j=1}^{\frac{p-1}{2}}\frac{(-1)^j}{j}\binom{\frac{p-3}{2}}{j-1}\\\label{yzhp14}&=\frac{2}{p-1}\sum_{j=1}^{\frac{p-1}{2}}\binom{\frac{p-1}{2}}{j}(-1)^j=\frac{-2}{p-1}\equiv 2\pmod{p}.\end{align}
Therefore (\ref{yzhp10}) with $a=1$ is proved by (\ref{yzhp12}), (\ref{jianjiankunkunyong}) and (\ref{yzhp14}).

By the above, we have completed the proof of Lemma 2.6. 
\end{proof}

\section{Proofs of Theorems 1.1 and 1.2}
\setcounter{equation}{0}
\setcounter{Thm}{0}
\setcounter{Lem}{0}
\setcounter{Cor}{0}
\setcounter{Conj}{0}
\begin{proof}[Proof of Theorem 1.1]
We first prove (\ref{jiankun12}). Let $\nu_{p}(n)=a.$ (\ref{jiankun12}) is evidently trivial when $a=0$. Next we suppose that $a\geq 1$.
With the help of (\ref{y71}),  we
have
\begin{align}\notag
\sum_{k=0}^{n-1}\binom{hn-1}{k}\binom{2k}{k}\bigg(-\frac{h}{2}\bigg)^k&\equiv \sum_{k=0}^{n-1}\binom{\frac{n}{p}h-1}{\floor {\frac{k}{p}}}(-1)^{k-\floor {\frac{k}{p}}}\binom{2k}{k}\bigg(-\frac{h}{2}\bigg)^k\\\notag&=\sum_{l=0}^{\frac{n}{p}-1}\binom{\frac{n}{p}h-1}{l}(-1)^{l}\sum_{\floor {\frac{k}{p}}=l}\binom{2k}{k}\bigg(\frac{h}{2}\bigg)^k\pmod{p^a}.
\end{align} 
By (\ref{yzhp5}) and (\ref{yzhp15}), for any odd prime $p$, then we get
$$\sum_{\floor {\frac{k}{p}}=l}\binom{2k}{k}\bigg(\frac{h}{2}\bigg)^k\equiv 0\pmod{p}.$$
Repeating the above process $a-1$ times, we obtain that
\begin{align}\notag
&\sum_{l=0}^{\frac{n}{p}-1}\binom{\frac{n}{p}h-1}{l}(-1)^{l}\sum_{\floor {\frac{k}{p}}=l}\binom{2k}{k}\bigg(\frac{h}{2}\bigg)^k\\\notag&\equiv\sum_{m=0}^{\frac{n}{p^2}-1}\binom{\frac{n}{p^2}h-1}{m}(-1)^{m}\sum_{\floor {\frac{l}{{p^{}}}}=m}\sum_{\floor {\frac{k}{{p^{}}}}=l}\binom{2k}{k}\bigg(\frac{h}{2}\bigg)^k\ \ \ (by\  (\ref{y71}))\\\notag&\equiv\sum_{l=0}^{\frac{n}{p^a}-1}\binom{\frac{n}{p^a}h-1}{l}(-1)^{l}\sum_{\floor {\frac{k}{p^a}}=l}\binom{2k}{k}\bigg(\frac{h}{2}\bigg)^k\equiv 0\pmod{p^a}.
\end{align} 
Assuming that (\ref{jiankun11}) holds, then
we prove (\ref{jiankun12}).

Let us turn to (\ref{jiankun11}). We assume that $p\geq 5$. In view of (\ref{y71}), we obtain
\begin{align}\notag&\sum_{k=0}^{p^a-1}\binom{hp^a-1}{k}\binom{2k}{k}\bigg(-\frac{h}{2}\bigg)^k\\\notag&\equiv\sum_{k=0}^{p^a-1}\binom{hp^{a-1}-1}{\floor {\frac{k}{p}}}(-1)^{k-\floor {\frac{k}{p}}}\bigg(1-hp^a\sum_{\substack{j=1\\p\nmid j}}^{k}\frac{1}{j}\bigg)\binom{2k}{k}\bigg(-\frac{h}{2}\bigg)^k\\\notag&=
\sum_{l=0}^{p^{a-1}-1}\bigg(\binom{hp^{a-1}-1}{l}(-1)^l-1\bigg)\sum_{\floor {\frac{k}{p}}=l}\binom{2k}{k}\bigg(\frac{h}{2}\bigg)^k(1-hp^a\sum_{\substack{j=1\\p\nmid j}}^{k}\frac{1}{j})\\\notag&\ \ +\sum_{l=0}^{p^{a-1}-1}\sum_{\floor {\frac{k}{p}}=l}\binom{2k}{k}\bigg(\frac{h}{2}\bigg)^k(1-hp^a\sum_{\substack{j=1\\p\nmid j}}^{k}\frac{1}{j})\pmod{p^{a+1}}.
\end{align} 
For any positive integer $a,$  we have $$p\mid \bigg(\binom{hp^{a-1}-1}{l}(-1)^l-1\bigg).$$
Therefore,
\begin{align}\notag&\sum_{k=0}^{p^a-1}\binom{hp^a-1}{k}\binom{2k}{k}\bigg(-\frac{h}{2}\bigg)^k\\\notag&\equiv\sum_{l=0}^{p^{a-1}-1}\binom{hp^{a-1}-1}{l}(-1)^l\sum_{\floor {\frac{k}{p}}=l}\binom{2k}{k}\bigg(\frac{h}{2}\bigg)^k\\\label{yzhp16}&\ \ -hp^a\sum_{l=0}^{p^{a-1}-1}\sum_{\floor {\frac{k}{p}}=l}\binom{2k}{k}\bigg(\frac{h}{2}\bigg)^k\sum_{\substack{j=1\\p\nmid j}}^{k}\frac{1}{j}\pmod{p^{a+1}}.
\end{align} 
With the help of Lemma 2.6, for $a\geq 2$ we get
\begin{align}\label{yzhp17}
\sum_{l=0}^{p^{a-1}-1}\sum_{\floor {\frac{k}{p}}=l}\binom{2k}{k}\bigg(\frac{h}{2}\bigg)^k\sum_{\substack{j=1\\p\nmid j}}^{k}\frac{1}{j}=\sum_{k=0}^{p^{a}-1}\binom{2k}{k}\bigg(\frac{h}{2}\bigg)^k\sum_{\substack{j=1\\p\nmid j}}^{k}\frac{1}{j}\equiv 0\pmod{p^{}}.
\end{align} 
If $a=1$, by (\ref{yzhp5}), (\ref{yzhp10}) and (\ref{yzhp16}), then
\begin{align}\notag&\sum_{k=0}^{p^a-1}\binom{hp^a-1}{k}\binom{2k}{k}\bigg(-\frac{h}{2}\bigg)^k\equiv\sum_{k=0}^{p-1}\binom{2k}{k}\bigg(\frac{h}{2}\bigg)^k-2hp\equiv p(1-2h)\equiv 0\pmod{p^{2}}.
\end{align} 
(\ref{jiankun11}) holds with $a=1$ and $p\geq 5$.
If $a\geq 2$, combining (\ref{yzhp16}) and (\ref{yzhp17}), then we have modulo $p^{a+1},$
\begin{align}\notag\sum_{k=0}^{p^a-1}\binom{hp^a-1}{k}\binom{2k}{k}\bigg(-\frac{h}{2}\bigg)^k&\equiv\sum_{l=0}^{p^{a-1}-1}\binom{hp^{a-1}-1}{l}(-1)^l\sum_{\floor {\frac{k}{p}}=l}\binom{2k}{k}\bigg(\frac{h}{2}\bigg)^k\\\label{yzhp20}&\equiv\sum_{l=0}^{p^{a-s}-1}\binom{hp^{a-s}-1}{l}(-1)^l\sum_{\floor {\frac{k}{p^s}}=l}\binom{2k}{k}\bigg(\frac{h}{2}\bigg)^k.
\end{align} 
Repeating this process $s-1$ times, we have
\begin{align}\notag&\sum_{l=0}^{p^{a-s}-1}\binom{hp^{a-s}-1}{l}(-1)^l\sum_{\floor {\frac{k}{p^s}}=l}\binom{2k}{k}\bigg(\frac{h}{2}\bigg)^k\\\notag&\equiv\sum_{m=0}^{p^{a-s-1}-1}\binom{hp^{a-s-1}-1}{m}(-1)^{m}\sum_{\floor {\frac{l}{p^{}}}=m}\sum_{\floor {\frac{k}{p^{s}}}=l}\binom{2k}{k}\bigg(\frac{h}{2}\bigg)^k\\\label{yzhp21}&\ \ -hp^{a-s}\sum_{m=0}^{p^{a-s-1}-1}\sum_{\floor {\frac{l}{p^{}}}=m}\sum_{\substack{j=1\\p\nmid j}}^{l}\frac{1}{j}\sum_{\floor {\frac{k}{p^{s}}}=l}\binom{2k}{k}\bigg(\frac{h}{2}\bigg)^k.
\end{align} 
By (\ref{yzhp5}) and (\ref{yzhp10}), we get 
\begin{align}\notag hp^{a-s}\sum_{l=1}^{p^{a-s}-1}\sum_{\substack{j=1\\p\nmid j}}^{l}\frac{1}{j}\sum_{\floor {\frac{k}{p^{s}}}=l}\binom{2k}{k}\bigg(\frac{h}{2}\bigg)^k&\equiv hp^{a}\sum_{l=1}^{p^{a-s}-1}\frac{\binom{2l}{l}}{4^l}\sum_{\substack{j=1\\p\nmid j}}^{l}\frac{1}{j}\\\notag&\equiv hp^{a}\sum_{l=1}^{p^{a-s}-1}\binom{2l}{l}\bigg(\frac{h}{2}\bigg)^l\sum_{\substack{j=1\\p\nmid j}}^{l}\frac{1}{j}\\\label{yzhp22}&\equiv\begin{cases}2hp^{a}\pmod{p^{a+1}},&\text{if }s=a-1,\\
0\pmod{p^{a+1}},&\text{if}\ s< a-1.\end{cases}
\end{align} 
At last, we only need to think about the case $s=a-1$. From (\ref{yzhp20})-(\ref{yzhp22}), we have
\begin{align}\notag\sum_{k=0}^{p^a-1}\binom{hp^a-1}{k}\binom{2k}{k}\bigg(-\frac{h}{2}\bigg)^k&\equiv \sum_{k=0}^{p^a-1}\binom{2k}{k}\bigg(\frac{h}{2}\bigg)^k-2hp^a\\\notag&\equiv (1-2h)p^a\equiv 0\pmod{p^{a+1}}.
\end{align} 
So we obtain (\ref{jiankun11}) with $a\geq 2$ and $p\geq 5$.
When $a\geq 2$ and $p=3,$
the proof of (\ref{jiankun11}) is very similar with the case $p\geq 5$, only requiring a few additional discussions. 
Without loss of generality, it suffices to study the following sum in (\ref{yzhp21}). Note that
\begin{align}\notag&h3^{a-s}\sum_{m=0}^{3^{a-s-1}-1}\sum_{\floor {\frac{l}{3^{}}}=m}\sum_{\substack{j=1\\3\nmid j}}^{l}\frac{1}{j}\sum_{\floor {\frac{k}{3^{s}}}=l}\binom{2k}{k}\bigg(\frac{h}{2}\bigg)^k\\\notag&=h3^{a-s}\sum_{m=0}^{3^{a-s-1}-1}\sum_{l=3m}^{3m+2}\sum_{\substack{j=1\\3\nmid j}}^{l}\frac{1}{j}\sum_{\floor {\frac{k}{3^{s}}}=l}\binom{2k}{k}\bigg(\frac{h}{2}\bigg)^k\\\notag&\equiv
h3^{a-s}\sum_{m=0}^{3^{a-s-1}-1}\bigg(\sum_{\substack{j=1\\3\nmid j}}^{3m+1}\frac{1}{j}\sum_{\floor {\frac{k}{3^{s}}}=3m+1}\binom{2k}{k}\bigg(\frac{h}{2}\bigg)^k\\\notag&\ \ +\sum_{\substack{j=1\\3\nmid j}}^{3m+2}\frac{1}{j}\sum_{\floor {\frac{k}{3^{s}}}=3m+2}\binom{2k}{k}\bigg(\frac{h}{2}\bigg)^k\bigg)\pmod{3^{a+1}}
\end{align} 
and $$\sum_{\substack{j=1\\3\nmid j}}^{3m+1}\frac{1}{j}=\sum_{\substack{j=1\\3\nmid j}}^{3m-1}\frac{1}{j}+\frac{1}{3m+1}\equiv 1\pmod{3^{}},\ \ \sum_{\substack{j=1\\3\nmid j}}^{3m+2}\frac{1}{j}\equiv1+\frac{1}{2}= \frac{3}{2}\equiv 0\pmod{3^{}}.$$
With the help of (\ref{yzhp8}), we obtain 
\begin{align}\notag&h3^{a-s}\sum_{m=0}^{3^{a-s-1}-1}\sum_{\floor {\frac{l}{3^{}}}=m}\sum_{\substack{j=1\\3\nmid j}}^{l}\frac{1}{j}\sum_{\floor {\frac{k}{3^{s}}}=l}\binom{2k}{k}\bigg(\frac{h}{2}\bigg)^k\\\notag&\equiv h3^{a-s}\sum_{m=0}^{3^{a-s-1}-1}\sum_{\floor {\frac{k}{3^{s}}}=3m+1}\binom{2k}{k}\bigg(\frac{h}{2}\bigg)^k\\\label{yzhp23}&\equiv h3^{a-s}\sum_{m=0}^{3^{a-s-1}-1}\bigg(3^s\frac{\binom{6m+2}{3m+1}}{4^{3m+1}}+23^{s-1}\frac{1-2h}{h}\binom{6m+2}{3m+1}\bigg)\pmod{3^{a+1}}.
\end{align} 
 In light of the Lucas' theorem, 
\begin{align}\label{yzhp24}\frac{\binom{6m+2}{3m+1}}{4^{3m+1}}=\frac{2(6m+1)}{4^{3m+1}(3m+1)}\binom{6m}{3m}\equiv2\binom{2m}{m}\pmod{3^{}}.\end{align}
Recall that Strauss et al. \cite{SMSJZD} proved that for any positive integer $n$,
\begin{align}\label{yzhp25}\nu_{3}\bigg(\sum_{k=0}^{n-1}\binom{2k}{k}\bigg)=2\nu_{3}(n)+\nu_{3}\bigg(\binom{2n}{n}\bigg).
\end{align}
Combining (\ref{yzhp23}), (\ref{yzhp24}) with (\ref{yzhp25}), then we have modulo $\pmod{3^{a+1}},$
\begin{align}\notag& h3^{a-s}\sum_{m=0}^{3^{a-s-1}-1}\sum_{\floor {\frac{l}{3^{}}}=m}\sum_{\substack{j=1\\3\nmid j}}^{l}\frac{1}{j}\sum_{\floor {\frac{k}{3^{s}}}=l}\binom{2k}{k}\bigg(\frac{h}{2}\bigg)^k\\\label{yzhp26}&\ \ \equiv \frac{2(2-h)3^a}{3}\sum_{m=0}^{3^{a-s-1}-1}\binom{2m}{m}\equiv\begin{cases}0,&\text{if }s\leq a-2,\\
\frac{2(2-h)3^a}{3},&\text{if}\ s=a-1.\end{cases}
\end{align}
Since $2h\equiv 1\equiv -2\pmod{3^{}},$ then $h\equiv 2\pmod{3^{}}$.
From (\ref{yzhp20}), (\ref{yzhp21}) and (\ref{yzhp26}), then
\begin{align}\notag\sum_{k=0}^{3^a-1}\binom{h3^a-1}{k}\binom{2k}{k}\bigg(-\frac{h}{2}\bigg)^k&\equiv\sum_{k=0}^{3^{a}-1}\binom{2k}{k}\bigg(\frac{h}{2}\bigg)^k-\frac{2(2-h)3^a}{3}\\\notag&\equiv\frac{(2-h)3^a}{3h}-\frac{2(2-h)3^a}{3}\ \ \ \ (by\ (\ref{yzhp4})) \\\notag&=\frac{(1-2h)(2-h)3^a}{3h}\equiv0\pmod{3^{a+1}}.
\end{align}

The proof of Theorem 1.1 is now complete.
\end{proof}
\begin{proof}[Proof of Theorem 1.2]Let $n=p^{a-1}d$
with $p\nmid d$ and $a\geq 1$. The proof of (\ref{jiankun1177}) is very similar with (\ref{jiankun12}). With the help of (\ref{y71}), for any integer $s\in [0,a-1]$ and nonnegative integer $k$ now we have
\begin{align}\label{47}
(-1)^k\binom{p^{a-s}d-1}{k}-(-1)^{\floor {\frac{k}{p^{}}}}\binom{p^{a-s-1}d-1}{{\floor {\frac{k}{p^{}}}}}\in p^{a-s}d\Z_{p}.
\end{align}
Therefore,
\begin{align}\notag&\frac{1}{d}\bigg(\sum_{k=0}^{p^{a}d-1}\binom{p^{a}d-1}{k}\frac{\binom{2k}{k}}{(-m)^k}-\bigg(\frac{\Delta}{p}\bigg)\sum_{k=0}^{p^{a-1}d-1}\binom{p^{a-1}d-1}{k}\frac{\binom{2k}{k}}{(-m)^k}\bigg)\\\notag
&\equiv\frac{1}{d}\sum_{l=0}^{p^{a-1}d-1}\binom{p^{a-1}d-1}{l}(-1)^l\bigg(\sum_{\floor {\frac{k}{p^{}}}=l}\frac{\binom{2k}{k}}{m^k}-\bigg(\frac{\Delta}{p}\bigg)\frac{\binom{2l}{l}}{m^l}\bigg)\pmod{p^{a}}.
\end{align}
In light of Lemma 2.5 with $s=1$, we have
\begin{align}\notag&\frac{1}{d}\sum_{l=0}^{p^{a-1}d-1}\binom{p^{a-1}d-1}{l}(-1)^l\bigg(\sum_{\floor {\frac{k}{p^{}}}=l}\frac{\binom{2k}{k}}{m^k}-\bigg(\frac{\Delta}{p}\bigg)\frac{\binom{2l}{l}}{m^l}\bigg)\\\notag
=&\frac{1}{d}\sum_{l=0}^{p^{a-1}d-1}\bigg(\binom{p^{a-1}d-1}{l}(-1)^l-\binom{p^{a-2}d-1}{\floor {\frac{l}{p^{}}}}(-1)^{\floor {\frac{l}{p^{}}}}\bigg)\bigg(\sum_{\floor {\frac{k}{p^{}}}=l}\frac{\binom{2k}{k}}{m^k}-\bigg(\frac{\Delta}{p}\bigg)\frac{\binom{2l}{l}}{m^l}\bigg)\\\notag&\ \ +\frac{1}{d}\sum_{l=0}^{p^{a-2}d-1}\binom{p^{a-2}d-1}{l}(-1)^l\bigg(\sum_{k=p^2l}^{p^2l+p^2-1}\frac{\binom{2k}{k}}{m^k}-\bigg(\frac{\Delta}{p}\bigg)\sum_{k=pl}^{pl+p-1}\frac{\binom{2l}{l}}{m^l}\bigg).
\end{align}
By Lemma 2.5 and (\ref{47}) we get
\begin{align*}&\frac{1}{d}\sum_{l=0}^{p^{a-1}d-1}\binom{p^{a-1}d-1}{l}(-1)^l\bigg(\sum_{\floor {\frac{k}{p^{}}}=l}\frac{\binom{2k}{k}}{m^k}-\bigg(\frac{\Delta}{p}\bigg)\frac{\binom{2l}{l}}{m^l}\bigg)\\\notag
&\equiv\frac{1}{d}\sum_{l=0}^{p^{a-2}d-1}\binom{p^{a-2}d-1}{l}(-1)^l\bigg(\sum_{k=p^2l}^{p^2l+p^2-1}\frac{\binom{2k}{k}}{m^k}-\bigg(\frac{\Delta}{p}\bigg)\sum_{k=pl}^{pl+p-1}\frac{\binom{2l}{l}}{m^l}\bigg)\\&
\equiv\frac{1}{d}\sum_{l=0}^{d-1}\binom{d-1}{l}(-1)^l\bigg(\sum_{k=p^al}^{p^al+p^a-1}\frac{\binom{2k}{k}}{m^k}-\bigg(\frac{\Delta}{p}\bigg)\sum_{k=p^{a-1}l}^{p^{a-1}l+p^{a-1}-1}\frac{\binom{2l}{l}}{m^l}\bigg)\equiv 0\pmod{p^{a}},
\end{align*}
where the last result comes from Lemma 2.5. 

In view of the above, we have completed the proof of Theorem 1.2. 
\end{proof}
 
\begin{Ack}
The author would like to thank the referee for helpful comments.
\end{Ack}

\end{document}